\documentclass[11pt]{amsart}
\usepackage[T1]{fontenc}
\usepackage[utf8]{inputenc}
\usepackage{lmodern}
\usepackage{microtype}
\usepackage{xcolor}
\usepackage{amsmath,amssymb,amsthm,mathtools}
\usepackage{enumitem}
\usepackage[margin=1in]{geometry}
\usepackage{hyperref}
\hypersetup{colorlinks=true,linkcolor=blue,citecolor=blue,urlcolor=blue}
\usepackage{mathrsfs}

\theoremstyle{plain}
\newtheorem{theorem}{Theorem}[section]
\newtheorem{lemma}[theorem]{Lemma}
\newtheorem{proposition}[theorem]{Proposition}
\newtheorem{corollary}[theorem]{Corollary}
\theoremstyle{remark}
\newtheorem{remark}[theorem]{Remark}

\DeclareMathOperator{\rank}{rank}
\DeclareMathOperator{\diag}{diag}
\newcommand{\FF}{\mathbb{F}}
\numberwithin{equation}{section}

\newcommand{\PP}{\mathscr{P}}
\newcommand{\J}{\mathcal{J}}
\DeclareMathOperator{\tr}{tr}
\DeclareMathOperator{\Span}{span}
\DeclareMathOperator{\End}{End}
\DeclareMathOperator{\im}{im}
\DeclareMathOperator{\rad}{rad}
\newcommand{\ceil}[1]{\left\lceil #1\right\rceil}

\hypersetup{pdftitle={Idempotent words in finite-dimensional simple Jordan algebras},pdfauthor={Ilja Gogic and Mateo Tomasevic},pdfsubject={Idempotent words in simple Jordan algebras}}

\title[Idempotent words in simple Jordan algebras]{Idempotent words in finite-dimensional simple Jordan algebras}

\author{Ilja Gogi\'{c}}
\address{I.~Gogi\'c, University of Zagreb Faculty of Science, Department of Mathematics, Bijeni\v{c}ka 30, 10000 Zagreb, Croatia}
\email{ilja@math.hr, ilja.gogic@math.pmf.unizg.hr}

\author{Mateo Toma\v{s}evi\'{c}}
\address{M.~Toma\v{s}evi\'c, University of Zagreb Faculty of Science, Department of Mathematics, Bijeni\v{c}ka 30, 10000 Zagreb, Croatia}
\email{mateo.tomasevic@math.hr, mateo.tomasevic@math.pmf.unizg.hr}

\date{\today}



\begin{document}
	
	\begin{abstract}
		We study finite left-nested Jordan products of idempotents in finite-dimensional simple Jordan algebras over a field $\mathbb{F}$ of characteristic different from $2$. Over an algebraically closed field, the structure theorem reduces the problem to the scalar algebra, spin factors, the full matrix Jordan algebras $M_n(\mathbb{F})^+$, the symmetric matrix algebras $H_n(\mathbb{F})$, the symplectic Hermitian algebras $H_{2n}(\mathbb{F},\mathrm{sp})$, and the split Albert algebra. We determine exactly which elements are representable: except in $H_2(\mathbb{F})$, these are all nonscalar elements together with $0$ and $1$, whereas in $H_2(\mathbb{F})$ the nonzero square-zero elements are also not representable. We also study the minimal word length of representable elements. For $M_n(\mathbb{F})^+$ over an arbitrary field, the maximal word length grows logarithmically in $n$, and for $n=2$ we determine the exact length of every nonscalar matrix. The same logarithmic growth holds for $H_{2n}(\mathbb{F},\mathrm{sp})$. The spin-factor and symmetric-matrix results hold over quadratically closed fields, whereas the full matrix, symplectic Hermitian, and split Albert results hold over arbitrary fields of characteristic different from $2$.
	\end{abstract}
	
	\maketitle
	
	\section{Introduction}
	
	Products of idempotent matrices have been studied since the classical work of Erd\H{o}s~\cite{Erdos1967}, who proved that every singular square matrix over a field $\FF$ is a product of idempotents. Ballantine~\cite{Ballantine1978} obtained the exact quantitative criterion: writing $M_n(\FF)$ for the algebra of all $n\times n$ matrices over $\FF$ and $I_n$ for its identity matrix, a matrix $A\in M_n(\FF)$ is a product of $k$ idempotent matrices, $k\ge1$, if and only if
	\[
	\rank(I_n-A)\le k\left(n-\rank A\right).
	\]
	Thus, for singular $A$, the minimum number of factors is determined by this inequality; in particular, every singular $n\times n$ matrix is a product of at most $n$ idempotents, and $n$ factors are necessary in general. Dawlings~\cite{Dawlings1981} later showed that the same bound holds with all factors of rank $n-1$.
	
	The problem has also been studied in broader algebraic and operator-theoretic settings. Jain and Leroy surveyed idempotent factorizations for matrices over special classes of rings, as well as factorizations of nonunits in regular rings; see~\cite{JainLeroy2019}. Cossu and Tringali developed an abstract factorization framework and applied it to recover and refine classical results on idempotent factorizations; see~\cite{CossuTringali2024}. In operator theory, Dawlings determined the subsemigroup generated by idempotents among the bounded linear operators on a separable Hilbert space~\cite{Dawlings1983}, while Arias, Corach and Maestripieri characterized products of two idempotent operators on Hilbert spaces~\cite{AriasCorachMaestripieri2017}. More recently, Jain, Leroy and Singh studied the semigroup generated by idempotents in the algebra of bounded operators on a Banach space~\cite{JainLeroySingh2026}.
	
	The associative results above suggest an analogous problem in nonassociative
	algebras, where the bracketing of the factors must also be specified. We study
	this problem for Jordan algebras. Let $\J$ be a unital Jordan algebra over a
	field $\FF$ of characteristic different from $2$, with product $\circ$ and
	unit $1$. For idempotents $e_1,\ldots,e_m\in\J$, we consider the finite
	left-nested product
	\begin{equation}\label{eq:word}
		(\cdots((e_1\circ e_2)\circ e_3)\circ\cdots)\circ e_m,
		\qquad e_i^2=e_i,\quad m\ge1,
	\end{equation}
	where for $m=1$ the product is simply $e_1$. We call \eqref{eq:word} a
	\emph{(left-nested) idempotent word} and set
	\[
	\PP(\J):=\{x\in\J:x\text{ is the value of an idempotent word}\}.
	\]
	For $x\in\PP(\J)$, define
	\[
	\ell_\J(x):=\min\{m\ge1:x\text{ has an idempotent word of length }m\},
	\qquad
	w(\J):=\sup_{x\in\PP(\J)}\ell_\J(x),
	\]
	where the supremum is allowed to be infinite. We call $\ell_\J(x)$ the
	\emph{word length} of $x$ in $\J$, and $w(\J)$ the \emph{idempotent-word width} of $\J$. When the context is clear, we simply say \emph{length} and \emph{width}, and omit the subscript on $\ell$. Throughout, idempotent words are understood to be left-nested as in \eqref{eq:word}. By commutativity of $\circ$, every right-nested product is a left-nested word of the same length after reversing the order of the factors. Arbitrary bracketings are much more flexible; we return to them in Corollary~\ref{cor:arbitrary-bracketing}.
	
	The associative and Jordan factorization problems differ in an essential way, even for two factors. If $\mathcal B$ is an associative $\FF$-algebra, we write $\mathcal B^+$ for the Jordan algebra obtained by equipping the same vector space with the normalized symmetrized product
	\begin{equation}\label{eq:symmetrized-product}
		x\circ y:=\frac12(xy+yx), \qquad x,y\in\mathcal B.
	\end{equation}
	Since $x\circ x=x^2$ for all $x \in \mathcal B$, the algebras $\mathcal B$ and $\mathcal B^+$ have the same idempotents. In $M_n(\FF)$, however, a product of idempotents can be invertible
	only if every factor is $I_n$, whereas in $M_2(\FF)^+$ the idempotents
	\[
	E:=\begin{bmatrix}1&1\\0&0\end{bmatrix},
	\qquad
	F:=\begin{bmatrix}1&0\\1&0\end{bmatrix}
	\]
	satisfy
	\[
	E\circ F=
	\frac12\begin{bmatrix}3&1\\1&1\end{bmatrix},
	\qquad
	\det(E\circ F)=\frac12.
	\]
	Thus their Jordan product is invertible but different from $I_2$. For related results on invertible anti-commutators of idempotents, see Khurana and Lam~\cite{KhuranaLam2018}.
	
	Related generation problems involving idempotents have also been studied. For $n\ge2$, the algebra $M_n(\FF)$ is additively spanned by Jordan products of pairs of idempotents; see~\cite[proof of Corollary~2.4]{BresarGodoy2023}. Jordan algebras generated by two idempotents were studied by Rowen and Segev~\cite{RowenSegev2017}. In contrast, we ask which elements occur as values of a single word \eqref{eq:word} and, when they do, how many idempotent factors are required.
	
	The restriction to idempotent factors is substantial. For $n\ge2$ and $0\ne A\in M_n(\FF)$, every $B\in M_n(\FF)$ can be obtained from $A$ by successive Jordan multiplications by matrices in $M_n(\FF)$; see~\cite[Proposition~2.2]{GT2025}. Moreover, every $\FF$-linear endomorphism of $M_n(\FF)$ is a finite composition of Jordan multiplication operators
	of the form $X\mapsto C\circ X$, where $C\in M_n(\FF)$; see \cite[Theorem~1.1]{GKT2026}. By contrast, in every nonzero unital Jordan algebra $\J$, restricting the factors to idempotents already yields the scalar obstruction
	\[
	\PP(\J)\cap\FF1=\{0,1\}.
	\]
	
	When $\FF$ is algebraically closed, the standard structure theorem reduces
	the finite-dimensional simple case to six families: the scalar algebra,
	spin factors, full and symmetric matrix Jordan algebras, symplectic
	Hermitian algebras, and the split Albert algebra. We recall these models and
	the structure theorem in Section~\ref{sec:preliminaries}. With one
	low-dimensional exception, our first main theorem shows that the scalar
	obstruction above is the only one. To state the result, for $n\ge1$ we write
	\[
	H_n(\FF):=\{A\in M_n(\FF):A^t=A\}
	\]
	for the Jordan subalgebra of symmetric matrices in $M_n(\FF)^+$, where $A^t$ denotes the transpose of $A$.
	
	\begin{theorem}\label{thm:classification}
		Let $\FF$ be algebraically closed with $\operatorname{char}\FF\ne2$, and let
		$\J$ be a finite-dimensional simple Jordan algebra over $\FF$.
		\begin{enumerate}[label=\textup{(\alph*)}]
			\item If $\J\not\cong H_2(\FF)$, then
			\[
			\PP(\J)=(\J\setminus\FF1)\cup\{0,1\}.
			\]
			\item For $H_2(\FF)$,
			\[
			\PP(H_2(\FF))=\{0,I_2\}\cup\{A\in H_2(\FF)\setminus\FF I_2:A^2\ne0\}.
			\]
			Moreover, $w(H_2(\FF))=4$, and an element
			$A\in\PP(H_2(\FF))$ has length four if and only if there exists $N\in H_2(\FF)$ such that
			\[
			A=-\frac18 I_2+N, \qquad N\ne0,\quad N^2=0.
			\]
		\end{enumerate}
	\end{theorem}
	
	Although Theorem~\ref{thm:classification} is stated over an algebraically closed field, the individual families can be treated under weaker hypotheses. The scalar case is immediate and has width $1$. For spin factors, quadratic closedness suffices, and Theorem~\ref{thm:spin} gives width $4$ when the algebra has dimension $3$ and width $3$ in all higher dimensions. Over more general fields, the existence of short idempotent words already depends on the square classes of $\FF$, as Lemma~\ref{lem:spin-discriminant} shows. For symmetric matrices of order $n\ge3$, quadratic closedness likewise suffices. Over real closed fields, however, even representability may fail, as noted in Remark~\ref{rem:real-closed}. By contrast, the full matrix, symplectic Hermitian and split Albert cases require only $\operatorname{char}\FF\ne2$. Finally, Remark~\ref{rem:quadratic-scope} shows that quadratic closedness alone does not suffice to extend Theorem~\ref{thm:classification} to all finite-dimensional simple Jordan algebras.
	
	We first record the full matrix result, which gives a complete description of the representable elements together with quantitative bounds for the maximum length. We use $\lceil\cdot\rceil$ for the ceiling function.
	
	\begin{theorem}\label{thm:matrices}
		Let $\FF$ be any field with $\operatorname{char}\FF\ne2$, and let $n\ge2$.
		Then
		\begin{equation}\label{eq:matrix-classification}
			\PP(M_n(\FF)^+)
			=(M_n(\FF)\setminus\FF I_n)\cup\{0,I_n\},
		\end{equation}
		and
		\begin{equation}\label{eq:matrix-width}
			1+\ceil{\log_2 n} \le w(M_n(\FF)^+) \le 2+5\ceil{\log_2 n}.
		\end{equation}
		If $\FF$ is quadratically closed, the upper bound improves to
		\[
		w(M_n(\FF)^+)\le 4+3\ceil{\log_2 n}.
		\]
	\end{theorem}
	
	It follows from Theorem~\ref{thm:matrices} that $w(M_n(\FF)^+)$ grows on the order of $\log n$ over arbitrary fields of characteristic different from $2$. For $2\times2$ matrices, Proposition~\ref{prop:exact-two} determines the exact length of every nonscalar matrix. For symmetric matrices over quadratically closed fields and $n\ge3$, Theorem~\ref{thm:symmetric} gives
	\[
	1+\ceil{\log_2 n} \le w(H_n(\FF)) \le 4n-4.
	\]
	These bounds leave a gap between logarithmic and linear growth. For symplectic Hermitian algebras (see \eqref{eq:symplectic-model}) over arbitrary fields of characteristic different from $2$ and $n\ge2$, Theorem~\ref{thm:symplectic} gives
	\[
	1+\ceil{\log_2 n} \le w(H_{2n}(\FF,\mathrm{sp})) \le w(M_n(\FF)^+).
	\]
	Combined with Theorem~\ref{thm:matrices}, this shows that $w(H_{2n}(\FF,\mathrm{sp}))$ also grows on the order of $\log n$. Finally, Proposition~\ref{prop:albert} gives idempotent-word width at most $10$ for the split Albert algebra, improving to at most $7$ over quadratically closed fields.
	
	The paper is organized as follows. Section~\ref{sec:preliminaries} fixes the notation and terminology used throughout, describes the simple models needed below, states the corresponding structure theorem, and establishes elementary properties of idempotent words. Section~\ref{sec:spin} treats spin factors, while Section~\ref{sec:matrices} is devoted to full matrix algebras, including the exact result for $2\times2$ matrices. Section~\ref{sec:symmetric} treats symmetric matrices. Finally, Section~\ref{sec:remaining} considers symplectic Hermitian algebras and the split Albert algebra, completes the proof of Theorem~\ref{thm:classification}, and derives consequences for arbitrary bracketings and semisimple algebras in Corollaries~\ref{cor:arbitrary-bracketing} and~\ref{cor:semisimple}.
	
	\section{Preliminaries}\label{sec:preliminaries}
	
	Throughout, $\FF$ denotes a field of characteristic different from $2$; any stronger assumptions on $\FF$ will be stated explicitly when needed. We call $\FF$ \emph{quadratically closed} if every polynomial of degree two over $\FF$ has a root in $\FF$. Since $\operatorname{char}\FF\ne2$, this is equivalent to requiring that every element of $\FF$ be a square. Such a field is necessarily infinite.
	
	For $p,q\ge1$, let $M_{p,q}(\FF)$ denote the space of $p\times q$ matrices over $\FF$; when $p=q=n$, we simply write $M_n(\FF)$. For $A\in M_{p,q}(\FF)$, we write $A^t$ for its transpose. For $A\in M_n(\FF)$, $\tr A$ denotes its trace. As usual, $I_n$ denotes the identity matrix in $M_n(\FF)$ and $E_{ij}$ the standard matrix unit. If $A\in M_p(\FF)$ and $B\in M_q(\FF)$, we write $A\oplus B:=\diag(A,B) \in M_{p+q}(\FF)$, and for $\lambda\in\FF$ we denote by $[\lambda]$
	the $1\times1$ matrix with entry $\lambda$. A matrix $S\in M_n(\FF)$ is called \emph{orthogonal} if $S^tS=I_n$, and similarity by an orthogonal matrix is called \emph{orthogonal similarity}. Diagonalizability is always understood over $\FF$.
	
	For an $\FF$-vector space $V$, we write $\End_\FF(V)$ for the algebra of $\FF$-linear operators on $V$ and $\Span S$ for the linear span of a subset $S\subseteq V$. For subspaces $U,W\subseteq V$, the notation $V=U\oplus W$ means that $V$ is their direct sum. If $V=U\oplus W$, $A\in\End_\FF(U)$ and $B\in\End_\FF(W)$, then $A\oplus B\in\End_\FF(V)$ denotes the operator given by
	\[
	(A\oplus B)(u+w):=Au+Bw, \qquad u \in U, \, w \in W.
	\]
	On a one-dimensional summand, $[\lambda]$ also denotes multiplication by $\lambda\in\FF$. 
	
	A matrix $A\in M_n(\FF)$ is called \emph{cyclic} if there exists $u\in\FF^n$ such that the vectors $u,Au,\ldots,A^{n-1}u$ form a basis of $\FF^n$. We shall use rational canonical form in the following standard form (see, e.g., \cite[Chapter~7, section ``The Rational Canonical Form'']{Roman2008}): every $A\in M_n(\FF)$ is similar to a direct sum of companion matrices corresponding to monic nonconstant polynomials $f_1,\ldots,f_s\in\FF[t]$, where $\FF[t]$ denotes the polynomial ring in the indeterminate $t$ over $\FF$, satisfying $f_1\mid\cdots\mid f_s$. These are the \emph{invariant factors} of $A$. Each companion block is cyclic, and $f_s$ is the minimal polynomial of $A$.

	A \emph{Jordan algebra} over $\FF$ is an $\FF$-algebra $\J$ whose multiplication $\circ$ is commutative and satisfies the \emph{Jordan identity}
	\[
	x^2\circ(x\circ y)=x\circ(x^2\circ y), \qquad x,y\in\J.
	\]
	An $\FF$-linear subspace $\mathcal I\subseteq\J$ is an \emph{ideal} of $\J$ if
	$x\circ y\in\mathcal I$ for all $x\in\mathcal I$ and $y\in\J$. A Jordan algebra $\J$ is \emph{simple} if its product is not identically zero and its only ideals are $\{0\}$ and $\J$. Every finite-dimensional simple Jordan algebra over $\FF$ is automatically unital; see \cite[Part~I, \S\,3.10]{McCrimmon2004}. Throughout, all Jordan algebras considered below are assumed to be nonzero and unital. A Jordan homomorphism is an $\FF$-linear map preserving the Jordan product.
	
	Every Jordan algebra over $\FF$ is \emph{power-associative}, meaning that the subalgebra generated by any single element is associative; see \cite[Part~II, Theorem~5.2.2(1)]{McCrimmon2004}. Consequently, the powers of any element $x\in\J$ are independent of bracketing and are given recursively by
	\[
	x^1:=x,\qquad x^{k+1}:=x^k\circ x,\qquad k\ge1.
	\]
	An element $e\in\J$ is called an \emph{idempotent} if $e^2=e$, and it is called \emph{nontrivial} if $e\notin\{0,1\}$. An element $x\in\J$ is called \emph{nilpotent} if $x^k=0$ for some $k\ge1$, and \emph{scalar} if $x\in\FF1$. Whenever $M_n(\FF)$ is regarded as a Jordan algebra, it is equipped with the product \eqref{eq:symmetrized-product}. We refer to \cite{Jacobson1968,McCrimmon2004,Schafer1966} for standard background on Jordan algebras.
	
	\subsection{The simple Jordan algebras}\label{subsec:models}
	
	We now describe the simple Jordan algebras that will occur below. The models themselves are defined over the fields indicated; the assumption that $\FF$ is algebraically closed will be used only when we invoke the classification theorem.
	
	Let $V$ be a finite-dimensional $\FF$-vector space, and let $b:V\times V\to\FF$ be a nondegenerate symmetric bilinear form. We write
	\[
	q(v):=b(v,v), \qquad v \in V,
	\]
	for the associated quadratic form. For a subspace $W\subseteq V$, we write
	\[
	W^\perp:=\{v\in V:b(v,w)=0\text{ for all }w\in W\}
	\]
	for its orthogonal complement. For $v\in V$, we also write $v^\perp:=(\FF v)^\perp$.
	
	For subspaces $U,W\subseteq V$, we write $V=U\oplus_b W$ if $V=U\oplus W$ and $b(u,w)=0$ for all $u\in U$ and $w\in W$; this is called an \emph{orthogonal direct sum}. A basis $\{e_1,\ldots,e_n\}$ of $V$ is called \emph{orthonormal} if $b(e_i,e_j)=\delta_{ij}$ for all $1 \le i,j \le n$, where $\delta_{ij}$ denotes the Kronecker delta. An \emph{isometry} between symmetric bilinear spaces is a linear isomorphism preserving the bilinear forms.
	
	A nonzero vector $v\in V$ is called \emph{isotropic} if $q(v)=0$ and \emph{anisotropic} otherwise. A subspace $W\subseteq V$ is \emph{totally isotropic} if $q$ vanishes identically on $W$; equivalently, $b$ vanishes on $W\times W$. For an alternating bilinear form $a$ on $V$, a subspace $W\subseteq V$ is likewise called \emph{totally isotropic} if $a$ vanishes on $W\times W$.
	
	A basis $\{u,v\}$ of a two-dimensional subspace $U\subseteq V$ is called a \emph{hyperbolic basis} if
	\[
	q(u)=q(v)=0,\qquad b(u,v)=1.
	\]
	If $U\subseteq V$ is a nondegenerate two-dimensional subspace, then any nonzero isotropic vector $u\in U$ can be completed to a hyperbolic basis: choose $v_0\in U$ with $b(u,v_0)=1$ and set $v:=v_0-\frac12q(v_0)u$.
	
	If $Q$ is a quadratic form on $V$, its \emph{polar form} is denoted by $b_Q$ and is given by
	\[
	b_Q(u,v):=\frac{1}{2}\left(Q(u+v)-Q(u)-Q(v)\right), \qquad u,v\in V.
	\]
	The \emph{radical} and \emph{rank} of $Q$ are defined by
	\begin{align*}
		\rad Q&:=\{u\in V:b_Q(u,v)=0\text{ for all }v\in V\},\\
		\rank Q&:=\dim V-\dim\rad Q,
	\end{align*}
	respectively.
	
	The \emph{spin factor} associated with $(V,b)$ is
	\[
	\mathcal S(V,b):=\FF1\oplus V,
	\]
	with unit $1$ and product
	\begin{equation}\label{eq:spin-product}
		(\alpha1+u)\circ(\beta1+v):=(\alpha\beta+b(u,v))1+\alpha v+\beta u, \qquad \alpha, \beta \in \FF, \, u,v \in V.
	\end{equation}
	In what follows, we consider spin factors with $\dim V\ge2$.
	
	For $n\ge1$, we write $M_n(\FF)^+$ for $M_n(\FF)$ regarded as a Jordan algebra with the product \eqref{eq:symmetrized-product}. Its unit is $I_n$ and its dimension is $n^2$. The symmetric matrices form the Jordan subalgebra
	\[
	H_n(\FF):=\{A\in M_n(\FF):A^t=A\},
	\]
	which has unit $I_n$ and dimension $n(n+1)/2$.
	
	We shall also use the coordinate-free version of $H_n(\FF)$. If $V$ is equipped with a nondegenerate symmetric bilinear form $b$, set
	\begin{equation}\label{eq:selfadjoint-model}
		H(V):=\{A\in\End_{\FF}(V):b(Au,v)=b(u,Av)\text{ for all } u,v\in V\}.
	\end{equation}
	Thus $H(V)$ is the Jordan algebra of $b$-self-adjoint endomorphisms of $V$, with the symmetrized product \eqref{eq:symmetrized-product} and unit the identity operator $I$. If $\FF$ is quadratically closed and $n=\dim V$, then $V$ admits an orthonormal basis, and with respect to such a basis $H(V)$ identifies with $H_n(\FF)$. Isometries of symmetric bilinear spaces induce Jordan isomorphisms between the corresponding algebras $H(V)$.
	
	For $n\ge1$, let
	\[
	R_n:=\begin{bmatrix}
		0&I_n\\-I_n&0
	\end{bmatrix}\in M_{2n}(\FF).
	\]
	The \emph{symplectic Hermitian algebra} is
	\begin{equation}\label{eq:symplectic-model}
		H_{2n}(\FF,\mathrm{sp})
		:=\{X\in M_{2n}(\FF):X^tR_n=R_nX\}.
	\end{equation}
	Equivalently, it consists of the endomorphisms self-adjoint with respect to the nondegenerate alternating bilinear form represented by $R_n$. It is a Jordan subalgebra of $M_{2n}(\FF)^+$, with unit $I_{2n}$ and dimension $n(2n-1)$. A matrix $S\in M_{2n}(\FF)$ is called \emph{symplectic} if $S^tR_nS=R_n$; conjugation by such an $S$ preserves $H_{2n}(\FF,\mathrm{sp})$.
	
	Finally, let $\mathbb O_s$ denote the eight-dimensional split octonion algebra over $\FF$, with standard conjugation $a\mapsto\bar a$. The \emph{split Albert algebra} is
	\begin{equation}\label{eq:albert-model}
		\mathcal A:=H_3(\mathbb O_s)=\left\{
		\begin{bmatrix}
			\alpha&a&b\\
			\bar a&\beta&c\\
			\bar b&\bar c&\gamma
		\end{bmatrix}
		: \alpha,\beta,\gamma\in\FF,\ a,b,c\in\mathbb O_s \right\}.
	\end{equation}
	Its product is $X\circ Y=(XY+YX)/2$, where $XY$ is the usual matrix product with octonion entries. The algebra $\mathcal A$ has dimension $27$ and unit $I_3$; see \cite[Part~I, \S\,2.12]{McCrimmon2004}.
	
	We use the following consequence of the Renaissance Structure Theorem for finite-dimensional simple Jordan algebras over algebraically closed fields; see again \cite[Part~I, \S\,2.13]{McCrimmon2004}.
	
	\begin{theorem}[Classification over algebraically closed fields]\label{thm:structure}
		Let $\FF$ be an algebraically closed field of characteristic different from $2$. Up to isomorphism, the finite-dimensional simple Jordan algebras over $\FF$ are precisely the following:
		\begin{enumerate}[label=\textup{(\alph*)}]
			\item the scalar algebra $\FF$;
			\item a spin factor $\mathcal S(V,b)$, where $b$ is nondegenerate and symmetric and $\dim V\ge2$;
			\item $M_n(\FF)^+$, where $n\ge3$;
			\item $H_n(\FF)$, where $n\ge3$;
			\item $H_{2n}(\FF,\mathrm{sp})$, where $n\ge3$;
			\item the split Albert algebra $\mathcal A=H_3(\mathbb O_s)$.
		\end{enumerate}
	\end{theorem}
	
	The family in \textup{(e)} is the usual quaternionic Hermitian family written in matrix form: over $\FF$, the split quaternion algebra is $M_2(\FF)$ with its standard symplectic involution, and the corresponding Hermitian Jordan algebra is $H_{2n}(\FF,\mathrm{sp})$.
	
	The restrictions $n\ge3$ in \textup{(c)}--\textup{(e)} merely avoid repetitions among the low-dimensional cases. For $n=1$, the three matrix families reduce to the scalar algebra $\FF$. For $n=2$, the algebras $H_2(\FF)$, $M_2(\FF)^+$ and $H_4(\FF,\mathrm{sp})$ are spin factors with vector parts of dimensions $2$, $3$ and $5$, respectively, and are therefore already covered by \textup{(b)}. The first two identifications are made explicit in the proof of Corollary~\ref{cor:two-by-two}, while the third follows from the split-quaternion description above together with \cite[Part~II, Example~11.1.4]{McCrimmon2004}.
	
	\subsection{Elementary properties of idempotent words}\label{subsec:word-properties}
	
	Let $\J$ be a Jordan algebra and let $e\in\J$ be an idempotent. The \emph{Peirce decomposition} of $\J$ with respect to $e$ is
	\[
	\J=\J_1(e)\oplus\J_{1/2}(e)\oplus\J_0(e),   \qquad
	\J_i(e):=\{x\in\J:e\circ x=ix\},
	\quad i\in\left\{0,\frac12,1\right\};
	\]
	see \cite[Part~II, 8.1.2 and 8.1.4]{McCrimmon2004}. Here $1=e+(1-e)$, with $e\in\J_1(e)$ and $1-e\in\J_0(e)$.
	
	\begin{lemma}\label{lem:scalars}
		For every unital Jordan algebra $\J$, we have $\PP(\J)\cap\FF1=\{0,1\}$.
	\end{lemma}
	
	\begin{proof}
		Suppose $\lambda1\in\PP(\J)$ with $\lambda\ne0$, and choose a shortest idempotent word for it. If its length is greater than one, write $\lambda1=y\circ e$, where $e$ is the last idempotent. The $\J_0(e)$-component of $y\circ e$ is zero, whereas that of $\lambda1$ is $\lambda(1-e)$. Hence $e=1$, contradicting the minimality of the word. Thus $\lambda1$ is itself an idempotent, so $\lambda=1$. Both $0$ and $1$ are idempotents.
	\end{proof}
	
	Jordan homomorphisms preserve idempotent words. In particular, if $\J$ is a Jordan subalgebra of a Jordan algebra $\mathcal K$, then $\ell_{\mathcal K}(x)\le\ell_{\J}(x)$ for every $x\in\PP(\J)$, while Jordan isomorphisms preserve length. Thus $\ell_{M_n(\FF)^+}$ is invariant under similarity and $\ell_{H_n(\FF)}$ under orthogonal similarity.
	
	Appending the unit does not change the value of an idempotent word. We shall use the following consequence.
	
	\begin{lemma}\label{lem:products}
		Let $\J_1,\ldots,\J_s$ be unital Jordan algebras. Then
		\[
		\PP\left(\prod_{i=1}^s\J_i\right)=\prod_{i=1}^s\PP(\J_i),
		\]
		and, for $x_i\in\PP(\J_i)$,
		\[
		\ell_{\prod_i\J_i}((x_i)_{i=1}^s)=\max_{1\le i\le s}\ell_{\J_i}(x_i).
		\]
	\end{lemma}
	
	\begin{proof}
		Projection to each factor gives one inequality. For the converse, choose shortest words for the $x_i$, append units until all have the same length, and multiply coordinatewise.
	\end{proof}
	
	In particular, a block diagonal matrix whose blocks have idempotent words has length at most the maximum of their lengths. The next lemma allows us to add arbitrary diagonal coordinates.
	
	\begin{lemma}\label{lem:padding}
		For $n\ge1$, let $\mathcal K_n$ denote either $M_n(\FF)^+$ or $H_n(\FF)$, with the same choice made for every $n$. Suppose that there is an integer $c\ge0$ such that, for every $\delta\in\FF$,
		\[
		\diag(\delta,1)\in\PP(\mathcal K_2), \qquad \ell_{\mathcal K_2}(\diag(\delta,1))\le c.
		\]
		Let $r\ge1$ and $s\ge0$, and let $X\in\PP(\mathcal K_r)$ have length at most $L$. If $s\ge1$ and $D$ is a diagonal matrix of order $s$, then
		\begin{equation}\label{eq:padding}
			\ell_{\mathcal K_{r+s}}(D\oplus X) \le L+c\ceil{\log_2\frac{r+s}{r}}.
		\end{equation}
		For $s=0$, the assertion reduces to $\ell_{\mathcal K_r}(X)\le L$.
	\end{lemma}
	
	\begin{proof}
		The case $s=0$ is immediate. First assume $1\le s\le r$. Pair each of the $s$ coordinates of $D$ with a different coordinate of $I_r$. Up to permutation, $D\oplus I_r$ is a direct sum of matrices $\diag(\delta_i,1)$ and copies of $[1]$. By Lemma~\ref{lem:products}, $\ell(D\oplus I_r)\le c$.
		
		Choose an idempotent word for $X$ with factors $E_1,\ldots,E_m$, where $m \le L$. Append $I_s\oplus E_1,\ldots,I_s\oplus E_m$ to a word of length at most $c$ for $D\oplus I_r$. The first block remains $D$, and the second block becomes $X$. Hence $\ell(D\oplus X)\le L+c$.
		
		For general $s$, add the coordinates of $D$ from right to left in successive blocks. Start with size $r_0:=r$ and set $r_{j+1}:=\min\{2r_j,r+s\}$. At step $j$, the number $r_{j+1}-r_j$ of new coordinates is at most the current size $r_j$, so the preceding argument applies. Induction yields $r_j = \min\{2^jr, r+s\}$, with the process terminating as soon as $2^j r \ge r+s$. Thus $\ceil{\log_2((r+s)/r)}$ steps suffice. Permutation matrices are orthogonal, so the same argument applies to symmetric idempotents.
	\end{proof}
	
	\section{Spin factors}\label{sec:spin}
	
	We first assume that $\FF$ is quadratically closed, and let $\mathcal S(V,b)$ be the spin factor defined in Section~\ref{subsec:models}, with $\dim V\ge2$. Formula~\eqref{eq:spin-product} shows that, besides $0$ and $1$, its idempotents are precisely
	\[
	e_t:=\frac12 1+t,
	\qquad \text{where } \, t\in V \text{ satisfies } \, q(t)=\frac14.
	\]
	If $W\subseteq V$ is a totally isotropic subspace, then $W\subseteq W^\perp$, and hence $2\dim W\le\dim V$. Consequently, for every $0\ne v\in V$, the hyperplane $v^\perp$ contains an anisotropic vector if $\dim V\ge3$. If $\dim V=2$ and $q(v)\ne0$, then $v^\perp$ is a nondegenerate line and hence also contains an anisotropic vector. As $\FF$ is quadratically closed, a quadratic
	form that represents a nonzero scalar represents every scalar, by scaling.
	
	\begin{lemma}\label{lem:spin-two}
		Let $s\in V\setminus\{0\}$ and $\beta\in\FF$. The element $\beta1+s$ has an idempotent word of length at most two if and only if
		\[
		\beta=2q(s)  \quad\text{ and }\quad \left(\dim V\ge3\ \text{ or }\ q(s)\ne0\right).
		\]
	\end{lemma}
	
	\begin{proof}
		For nontrivial idempotents $e_u$ and $e_v$, with $u,v\in V$ satisfying $q(u)=q(v)=1/4$,
		\[
		e_u\circ e_v
		\stackrel{\eqref{eq:spin-product}}{=}\left(\frac14+b(u,v)\right)1+\frac{u+v}{2}.
		\]
		Thus $\beta1+s=e_u\circ e_v$ implies $\beta=2q(s)$. If $\beta1+s$ itself is a nontrivial idempotent, then $\beta=\frac12$ and $q(s)=\frac14$, so again $\beta=2q(s)$. Hence allowing a unit factor does not change this necessary condition.
		
		Conversely, suppose $\beta=2q(s)$. If $s^\perp$ contains an anisotropic vector, choose $z\in s^\perp$ with $q(z)=1/4-q(s)$. Then $e_{s+z}$ and $e_{s-z}$ are idempotents, and
		\[
		e_{s+z}\circ e_{s-z}\stackrel{\eqref{eq:spin-product}}{=}2q(s)1+s.
		\]
		This applies in the two cases stated in the lemma.
		
		Finally, suppose $\dim V=2$ and $s\ne0$ is isotropic. A representation by two nontrivial idempotents would give $u,v\in V$ satisfying $q(u)=q(v)=1/4$, with $s=(u+v)/2$ and $z:=(u-v)/2$. Then $b(s,z)=0$ and $q(z)=1/4$. But $s^\perp=\FF s$, on which $q$ vanishes. A unit factor would make $\beta1+s$ a nontrivial idempotent, which is impossible since $q(s)=0$.
	\end{proof}
	
	\begin{theorem}\label{thm:spin}
		Let $\FF$ be quadratically closed with $\operatorname{char}\FF\ne2$, and let $\mathcal S:=\mathcal S(V,b)$.
		\begin{enumerate}[label=\textup{(\alph*)}]
			\item If $\dim V\ge3$, then
			\[
			\PP(\mathcal S)=(\mathcal S\setminus\FF1)\cup\{0,1\}, \qquad w(\mathcal S)=3.
			\]
			\item If $\dim V=2$, then
			\[
			\PP(\mathcal S)=\{0,1\}\cup  \{\alpha1+v:\alpha\in\FF,\ 0\ne v\in V,\ (\alpha,q(v))\ne(0,0)\}.
			\]
			In this case $w(\mathcal S)=4$, and the elements of length four are precisely
			\[
			-\frac18 1+v, \qquad 0\ne v\in V,\quad q(v)=0.
			\]
		\end{enumerate}
	\end{theorem}
	
	\begin{proof}
		The scalar restriction follows from Lemma~\ref{lem:scalars}. Fix $\alpha\in\FF$ and $0\ne v\in V$.
		
		First suppose that $\dim V\ge3$ or $q(v)\ne0$. Choose $t_0 \in V$ with $b(v,t_0)=\alpha/2$, and choose an anisotropic vector $z_0\in v^\perp$. Since $q(z_0)\ne0$, the equation $q(t_0+\gamma z_0)=1/4$ is quadratic in $\gamma$ with nonzero leading coefficient. As $\FF$ is quadratically closed, choose $\gamma\in\FF$ satisfying it, and put $t:=t_0+\gamma z_0$. Then
		\begin{equation}\label{eq:spin-t}
			q(t)=\frac14, \qquad b(v,t)=\frac\alpha2.
		\end{equation}
		Let $\beta \in \FF$ be a root of
		\begin{equation}\label{eq:spin-quadratic}
			2\beta^2-(8\alpha+1)\beta+8q(v)=0,
		\end{equation}
		and set $s:=2v-2\beta t$. Then
		\[
		q(s)\stackrel{\eqref{eq:spin-t}}{=}4q(v)-4\alpha\beta+\beta^2
		\stackrel{\eqref{eq:spin-quadratic}}{=}\frac\beta2,
		\]
		while \eqref{eq:spin-product} and \eqref{eq:spin-t} give
		\begin{equation}\label{eq:spin-predecessor}
			(\beta1+s)\circ e_t\stackrel{\eqref{eq:spin-product}}{=}\alpha1+v.
		\end{equation}
		The identity \eqref{eq:spin-predecessor} depends only on \eqref{eq:spin-t}, so it remains valid for every $\beta\in\FF$ if $s:=2v-2\beta t$. For the chosen root $\beta$ of \eqref{eq:spin-quadratic}, the equality $q(s)=\beta/2$ shows that $s=0$ would force $\beta=0$ and then $v=0$. Hence $s\ne0$. If $\dim V\ge3$, Lemma~\ref{lem:spin-two} applies directly. If $q(v)\ne0$, the constant term in \eqref{eq:spin-quadratic} is nonzero, so $\beta\ne0$ and $q(s)\ne0$, and Lemma~\ref{lem:spin-two} applies again. Thus $\ell(\alpha1+v)\le 3$.
		
		This proves the upper bound in (a). To see that it is sharp, take $v\in V$ with $q(v)\ne0$. The element $v$ does not satisfy the necessary two-factor relation in Lemma~\ref{lem:spin-two}.
		
		It remains to consider $\dim V=2$ and a nonzero isotropic vector $v$. A final nontrivial idempotent in a representation of $\alpha1+v$ forces \eqref{eq:spin-t}: indeed, if $\alpha1+v=(\beta1+s)\circ e_t$, then
		\[
		v=\frac12s+\beta t, \qquad \alpha=\frac\beta2+b(s,t).
		\]
		If $\alpha=0$, this contradicts $v^\perp=\FF v$ and $q(t)=1/4$. Since $v^2=0\ne v$, the element $v$ is not an idempotent. If $v$ had an idempotent word, then after deleting any final factors equal to $1$, a word of length at least two would remain. Its last factor could not be $0$ because $v\ne0$, and hence would be a nontrivial idempotent, contradicting the preceding argument. Thus $v\notin\PP(\mathcal S)$.
		
		Suppose $\alpha\ne0$. Extend $\{v\}$ to a hyperbolic basis $\{v,z\}$, with $q(z)=0$ and $b(v,z)=1$. Then $t:=\frac{1}{4\alpha}v+\frac\alpha2 z$ satisfies \eqref{eq:spin-t}, and \eqref{eq:spin-quadratic} reduces to $\beta(2\beta-(8\alpha+1))=0$. If $\alpha\ne-1/8$, choose $\beta:=(8\alpha+1)/2\ne0$. The preceding construction again gives at most three factors.
		
		Let $\alpha=-1/8$. Lemma~\ref{lem:spin-two} excludes a representation with at most two factors. In a three-factor word, the product of the first two factors cannot be scalar: its only possible scalar values are $0$ and $1$, whereas $\alpha1+v$ is neither zero nor an idempotent. Write this two-factor product as $\beta1+s$. Lemma~\ref{lem:spin-two}, together with the final multiplication, forces \eqref{eq:spin-quadratic}, and hence $\beta=0$ and $s=2v$. Since $s=2v$ is nonzero and isotropic, this is impossible by
		Lemma~\ref{lem:spin-two}.
		
		To obtain four factors, keep the above $t$, choose $\beta\notin\{0,4\alpha\}$, and put $s:=2v-2\beta t$. Since $q(s)=\beta(\beta-4\alpha)\ne0$, the first part of the proof gives a word of length at most three for $\beta1+s$, and \eqref{eq:spin-predecessor} remains valid for this choice of $\beta$. Appending $e_t$ gives the required fourth factor. Since $\FF$ is quadratically closed, every nondegenerate two-dimensional symmetric bilinear space over $\FF$ contains a nonzero isotropic vector, so the bound four is attained.
	\end{proof}
	
	\begin{corollary}\label{cor:two-by-two}
		Over a quadratically closed field $\FF$ of characteristic different from $2$,
		\[
		\PP(M_2(\FF)^+)=(M_2(\FF)\setminus\FF I_2)\cup\{0,I_2\}, \qquad w(M_2(\FF)^+)=3.
		\]
		For $H_2(\FF)$,
		\[
		\PP(H_2(\FF))=\{0,I_2\}\cup\{A\in H_2(\FF)\setminus\FF I_2:A^2\ne0\}, \qquad w(H_2(\FF))=4,
		\]
		and an element $A\in\PP(H_2(\FF))$ has length four if and only if there exists $N\in H_2(\FF)$ such that
		\[
		A=-\frac18 I_2+N, \qquad N\ne0,\quad N^2=0.
		\]
		Every nonscalar diagonal $2\times2$ matrix has an idempotent word in $H_2(\FF)$ of length at most three. In particular,
		\begin{equation}\label{eq:diagonal-two}
			\ell_{H_2(\FF)}(\diag(\delta,1))\le 3,
			\qquad \delta\in\FF.
		\end{equation}
	\end{corollary}
	
	\begin{proof}
		For $A\in M_2(\FF)$, set $A_0:=A-\frac12\tr(A)I_2$, so that $\tr(A_0)=0$.
		If $B_0\in M_2(\FF)$ is also traceless, the Cayley–Hamilton theorem and polarization give
		\[
		A_0^2=-\det(A_0)I_2, \qquad A_0\circ B_0=\frac12\tr(A_0B_0)I_2.
		\]
		These identities hold over every field of characteristic different from $2$. Thus $M_2(\FF)^+$ is the spin factor on its traceless subspace, with bilinear form
		\[
		b(A_0,B_0):=\frac12\tr(A_0B_0).
		\]
		The associated quadratic form is
		\[
		q\left(\begin{bmatrix}\alpha&\beta\\\gamma&-\alpha\end{bmatrix}\right):=\alpha^2+\beta\gamma, \qquad \alpha,\beta,\gamma \in \FF.
		\]
		This quadratic form is nondegenerate and has dimension three. Its restriction to the traceless symmetric matrices is $\alpha^2+\beta^2$, a nondegenerate form of dimension two. Theorem~\ref{thm:spin} therefore gives both matrix descriptions and the stated length conclusions. In the symmetric model, the nonzero nilpotents are precisely the nonzero traceless matrices $N$ with $q(N)=0$, equivalently $N^2=0$.
		
		If $A=\diag(\alpha,\beta)$ with $\alpha\ne\beta$, its traceless part satisfies $q(A_0)=(\alpha-\beta)^2/4\ne0$. The three-factor construction in the proof of Theorem~\ref{thm:spin} applies. The case $\delta=1$ in \eqref{eq:diagonal-two} is the identity matrix.
	\end{proof}
	
	Without quadratic closedness, the quadratic equation \eqref{eq:spin-quadratic} need not have a root in $\FF$. The following lemma isolates the resulting discriminant obstruction to short idempotent words and will be used in Proposition~\ref{prop:exact-two}.
	
	\begin{lemma}\label{lem:spin-discriminant}
		Let $\FF$ be any field of characteristic different from $2$, and let $x=\alpha1+v\in\mathcal S(V,b)$ be nonscalar, where $\alpha\in\FF$ and $v\in V$. If $x$ has an idempotent word of length at most three, then
		\[
		D(x):=(8\alpha+1)^2-64q(v)
		\]
		is a square in $\FF$. Moreover, if $\theta\in\FF$ is not a square and $u\in V$ satisfies $q(u)=(1-\theta)/256$, then $-\frac1{16}1+u$ has no idempotent word of length at most four.
	\end{lemma}
	
	\begin{proof}
		If $x$ has length at most two, the calculation in the necessity part of the proof of Lemma~\ref{lem:spin-two}, which does not require $\FF$ to be quadratically closed, gives $\alpha=2q(v)$ and hence $D(x)=(8\alpha-1)^2$. Suppose now that $x$ has a shortest word of length three. Its last factor is nontrivial: it cannot be $0$ because $x$ is nonscalar, while a final unit could be deleted, contradicting minimality. Its two-factor predecessor is nonscalar: if it were scalar, Lemma~\ref{lem:scalars} would make it $0$ or $1$, forcing $x$ to have length one. Write
		\[
		x=(\beta1+s)\circ e_t,
		\qquad \beta\in\FF,\quad s,t\in V,\quad q(t)=\frac14.
		\]
		Comparing components gives
		\[
		s=2v-2\beta t,\qquad b(v,t)=\frac{\alpha}{2}.
		\]
		The necessary two-factor relation for the predecessor is $\beta=2q(s)$. Hence
		\[
		2\beta^2-(8\alpha+1)\beta+8q(v)=0,
		\]
		whose discriminant is $D(x)$. Thus $D(x)$ is a square in $\FF$.
		
		For the second assertion, put $\alpha:=-1/16$ and $x:=\alpha1+u$. Since $\theta$ is not a square, $\theta\ne1$, so $q(u)\ne0$ and hence $u\ne0$. Moreover,
		\[
		\alpha^2-q(u)=\frac{\theta}{256}\ne0,
		\]
		so $x$ is neither zero nor an idempotent. Suppose that $x$ has an idempotent word of length at most four. After deleting any trailing unit factors, we may assume that the last factor is nontrivial:
		\[
		x=(\beta1+s)\circ e_t,
		\qquad \beta\in\FF,\quad s,t\in V,\quad q(t)=\frac14.
		\]
		The predecessor is nonscalar: otherwise Lemma~\ref{lem:scalars} would make it $0$ or $1$, and $x$ would be zero or an idempotent. Comparing components gives
		\[
		s=2u-2\beta t,\qquad b(u,t)=\frac{\alpha}{2}.
		\]
		Thus
		\[
		q(s)=4q(u)-4\alpha\beta+\beta^2,
		\]
		and
		\[
		D(\beta1+s)
		=16(16\alpha+1)\beta+1-256q(u)
		=\theta.
		\]
		The predecessor has length at most three, contradicting the first assertion.
	\end{proof}
	
	\section{Full matrix algebras}\label{sec:matrices}
	
	Throughout this section, $\FF$ is an arbitrary field of characteristic different from $2$. Word lengths are always taken in the relevant full matrix Jordan algebra. Corollary~\ref{cor:two-by-two} already settles representability and the width of
	$M_2(\FF)^+$ when $\FF$ is quadratically closed. We now drop this hypothesis and determine the exact length of every nonscalar $2\times2$ matrix.
	
	\begin{proposition}\label{prop:exact-two}
		Let $A\in M_2(\FF)$ be a nonscalar matrix, and set
		\[
		\tau:=\tr A,\qquad \delta:=\det A,\qquad \eta:=1+8\tau+64\delta.
		\]
		Then $A\in\PP(M_2(\FF)^+)$ and
		\[
		\ell(A)=
		\begin{cases}
			1,&A^2=A,\\
			2,&A^2\ne A,\quad 4\delta=\tau(\tau-1),\\
			3,&4\delta\ne\tau(\tau-1),\quad\eta\text{ is a square in }\FF,\\
			4,&\eta\text{ is not a square in }\FF,\quad \tau\ne-1/8,\\
			5,&\eta\text{ is not a square in }\FF,\quad \tau=-1/8.
		\end{cases}
		\]
		Here zero is included among the squares. Consequently, $w(M_2(\FF)^+)=3$ if $\FF$ is quadratically closed, and $w(M_2(\FF)^+)=5$ otherwise.
	\end{proposition}
	
	\begin{proof}
		Put $\alpha:=\tau/2$ and $\gamma:=\alpha^2-\delta$. Since $A-\alpha I_2$ is a nonscalar $2\times2$ matrix, it is cyclic. Hence, after similarity,
		\begin{equation}\label{eq:two-normal-form}
			A=\begin{bmatrix}\alpha&\gamma\\1&\alpha\end{bmatrix}.
		\end{equation}
		Both the spin-factor identification established in the proof of Corollary~\ref{cor:two-by-two} and the necessity calculation in the proof of Lemma~\ref{lem:spin-two} are valid over every field of characteristic different from $2$. Under this identification, $A=\alpha I_2+(A-\alpha I_2)$ and $q(A-\alpha I_2)=\gamma$. If $A$ has length at most two, that calculation gives $\alpha=2\gamma$, equivalently $4\delta=\tau(\tau-1)$. This condition is also sufficient, since
		\[
		\begin{bmatrix}1&2\gamma\\0&0\end{bmatrix}
		\circ\begin{bmatrix}0&0\\2&1\end{bmatrix}
		=\begin{bmatrix}2\gamma&\gamma\\1&2\gamma\end{bmatrix} = A,
		\]
		and both factors are idempotent.
		
		For every $\beta\in\FF$, set
		\[
		B_\beta:=\begin{bmatrix}0&2\gamma-2\alpha\beta\\2&2\beta\end{bmatrix},
		\qquad E:=\begin{bmatrix}1&\alpha\\0&0\end{bmatrix}.
		\]
		Then $B_\beta$ is nonscalar, $E^2=E$ and $B_\beta\circ E=A$. By the criterion just obtained, $B_\beta$ has length at most two if and only if
		\[
		2\beta^2-(8\alpha+1)\beta+8\gamma=0.
		\]
		Its discriminant is $(8\alpha+1)^2-64\gamma=\eta$. Hence, if $\eta$ is a square, then $A$ has length at most three. Conversely, if $A$ has length at most three, Lemma~\ref{lem:spin-discriminant}, via the same identification, shows that $\eta$ is a square.
		
		Suppose that $\eta$ is not a square. For $B_\beta$, we have
		\begin{equation}\label{eq:two-length-discriminant}
			1+8\tr B_\beta+64\det B_\beta =1-256\gamma+16(16\alpha+1)\beta.
		\end{equation}
		If $\alpha\ne-1/16$, put $\beta:=(256\gamma-1)/(16(16\alpha+1))$, so that \eqref{eq:two-length-discriminant} is zero. The preceding case gives at most three factors for $B_\beta$, and thus four for $A$.
		
		If $\alpha=-1/16$, take $\beta:=0$. Then $B_0$ has trace zero and \eqref{eq:two-length-discriminant} gives
		\[
		1+8\tr B_0+64\det B_0=4\eta,
		\]
		which is not a square. Since $\frac12\tr(B_0)=0\ne-1/16$, the preceding paragraph gives $\ell(B_0)\le 4$, while Lemma~\ref{lem:spin-discriminant}, via the same identification, gives $\ell(B_0)>3$. Hence $\ell(B_0)=4$, and therefore $A$ has length at most five. Finally, $\gamma=(1-4\eta)/256$, so the second assertion of Lemma~\ref{lem:spin-discriminant}, applied with the nonsquare $4\eta$, shows that $A$ has no idempotent word of length at most four. Since $\tau=2\alpha$, the exceptional case $\alpha=-1/16$ is precisely $\tau=-1/8$. The one-factor condition is idempotence, and all five cases follow. If $\FF$ is quadratically closed, Corollary~\ref{cor:two-by-two} gives width three. Otherwise, choose a nonsquare $\eta\in\FF$ and take $\alpha:=-1/16$ and $\gamma:=(1-4\eta)/256$ in \eqref{eq:two-normal-form}; the resulting matrix has length five.
	\end{proof}
	
	\begin{lemma}\label{lem:block-identities}
		Let $B\in M_{p,q}(\FF)$ and $C\in M_{q,p}(\FF)$, where $p,q\ge1$.
		\begin{enumerate}[label=\textup{(\alph*)}]
			\item The matrix $\begin{bmatrix}0&B\\C&0\end{bmatrix}$ has an idempotent word of length at most four.
			\item For every $T\in M_{p,q}(\FF)$, the matrix $\begin{bmatrix}0&B\\C&CT\end{bmatrix}$ has an idempotent word of length at most six.
		\end{enumerate}
	\end{lemma}
	
	\begin{proof}
		For (a), consider the idempotents
		\[
		E_1:=\begin{bmatrix}I_p&0\\8C&0\end{bmatrix},\quad
		E_2:=\begin{bmatrix}0&8B\\0&I_q\end{bmatrix},\quad
		E_3:=\begin{bmatrix}I_p&0\\0&0\end{bmatrix},\quad
		E_4:=\begin{bmatrix}0&0\\0&I_q\end{bmatrix}.
		\]
		Direct multiplication gives
		\[
		E_1\circ E_2=
		\begin{bmatrix}32BC&4B\\4C&32CB\end{bmatrix},
		\qquad
		(E_1\circ E_2)\circ E_3=
		\begin{bmatrix}32BC&2B\\2C&0\end{bmatrix},
		\]
		and hence
		\[
		((E_1\circ E_2)\circ E_3)\circ E_4=
		\begin{bmatrix}0&B\\C&0\end{bmatrix}.
		\]
		For (b), the similarity
		\[
		\begin{bmatrix}I_p&T\\0&I_q\end{bmatrix}
		\begin{bmatrix}0&B\\C&CT\end{bmatrix}
		\begin{bmatrix}I_p&-T\\0&I_q\end{bmatrix}
		=\begin{bmatrix}TC&B\\C&0\end{bmatrix}
		\]
		reduces the problem to the matrix on the right. By (a), the matrix $\begin{bmatrix}0&4B\\4C&0\end{bmatrix}$ has an idempotent word of length at most four. Moreover,
		\[
		\left(\begin{bmatrix}0&4B\\4C&0\end{bmatrix}
		\circ\begin{bmatrix}I_p&T/2\\0&0\end{bmatrix}\right)
		\circ\begin{bmatrix}I_p&0\\0&0\end{bmatrix}
		=\begin{bmatrix}TC&B\\C&0\end{bmatrix}.
		\]
		The other two matrices are idempotent, so six factors suffice.
	\end{proof}
	
	\begin{proposition}\label{prop:cyclic}
		Every cyclic matrix of order $n\ge3$ has an idempotent word of length at most seven. Six factors suffice when $n$ is even or the matrix is singular.
	\end{proposition}
	
	\begin{proof}
		Let $A\in M_n(\FF)$ be cyclic. Choose a cyclic basis
		$v_1,\ldots,v_n$ of $\FF^n$ for $A$, with
		$Av_i=v_{i+1}$ for $i<n$, and write
		\[
		Av_n=\sum_{j=1}^n\alpha_jv_j,
		\qquad \alpha_1,\ldots,\alpha_n\in\FF.
		\]
		\emph{Case 1: $n=2k$.} Set
		\[
		U:=\Span\{v_1,v_3,\ldots,v_{2k-1}\},
		\qquad W:=\Span\{v_2,v_4,\ldots,v_{2k}\}.
		\]
		Relative to $U\oplus W$, the matrix of $A$ has the form
		\begin{equation}\label{eq:block-form}
			\begin{bmatrix}0&B\\C&D\end{bmatrix},
		\end{equation}
		where $C:U\to W$ is invertible. Since $D=C(C^{-1}D)$, Lemma~\ref{lem:block-identities}(b) gives six factors.
		
		\emph{Case 2: $n=2k+1$ and $A$ is singular.} In the cyclic basis above, $\det A=(-1)^{n-1}\alpha_1$, so singularity gives $\alpha_1=0$. Take
		\[
		U:=\Span\{v_2,v_4,\ldots,v_{2k}\},
		\qquad W:=\Span\{v_1,v_3,\ldots,v_{2k+1}\}.
		\]
		Again, $A$ has the block form \eqref{eq:block-form}, and
		\[
		\im C=\Span\{v_3,v_5,\ldots,v_{2k+1}\}.
		\]
		The map $D$ vanishes on $v_1,v_3,\ldots,v_{2k-1}$, while $Dv_{2k+1}\in\im C$ because $\alpha_1=0$. Hence $D=CT$ for some $T:W\to U$, and Lemma~\ref{lem:block-identities}(b) applies.
		
		\emph{Case 3: $n\ge3$ is odd and $A$ is invertible.}
		Fix a cyclic vector $u\in\FF^n$ for $A$. The vectors
		$u,Au,A^{-1}u$ are linearly independent, since a dependence,
		multiplied by $A$, would contradict the independence of
		$u,Au,A^2u$. Choose a linear functional
		$\varphi:\FF^n\to\FF$ satisfying
		\[
		\varphi(u)=1,
		\qquad \varphi(Au)=0,
		\qquad \varphi(A^{-1}u)=1.
		\]
		Define $P:\FF^n\to\FF^n$ by $P(x):=\varphi(x)u$, and set
		\[
		E:=I_n-P,\qquad Y:=A+PA+AP-4P.
		\]
		Since $P^2=P$ and $PAP=0$, we have $Y\circ E=A$. Moreover,
		\[
		Y\left(A^{-1}u-\frac12u\right)=0,
		\qquad \varphi\left(A^{-1}u-\frac12u\right)=\frac12,
		\]
		so $Y$ is singular. To see that $Y$ is cyclic, note that $Yu=2Au-4u$ and
		\[
		YA^ju-A^{j+1}u\in\Span\{u,Au\},\qquad j\ge1.
		\]
		Together these relations imply
		\[
		YA^iu\in\Span\{u,Au,\ldots,A^{i+1}u\},  \qquad 0\le i<n-1.
		\]
		An induction then gives
		\[
		Y^ju-2A^ju\in\Span\{u,Au,\ldots,A^{j-1}u\}, \qquad 1\le j<n.
		\]
		Thus the matrix whose columns are the coordinates of $u,Yu,\ldots,Y^{n-1}u$ with respect to the ordered basis $(u,Au,\ldots,A^{n-1}u)$ of $\FF^n$ is triangular with diagonal entries $1,2,\ldots,2$. Since $2\ne0$, it is invertible, so $u,Yu,\ldots,Y^{n-1}u$ form a basis of $\FF^n$. Case~2 gives a word of length at most six for $Y$. Appending $E$ gives a word of length at most seven for $A$.
	\end{proof}
	
	The lower bound follows from an elementary rank inequality.
	
	\begin{lemma}\label{lem:rank}
		Let $X,E\in M_n(\FF)$ with $E^2=E$, and let $\lambda\in \FF \setminus \{0\}$. Then
		\begin{equation}\label{eq:rank-step}
			\rank(X-\lambda I_n)
			\le 2\rank(X\circ E-\lambda I_n).
		\end{equation}
		Consequently, if $A \in M_n(\FF)$ has an idempotent word of length $m$ and $\lambda\notin\{0,1\}$, then
		\begin{equation}\label{eq:rank-product}
			n\le 2^{m-1}\rank(A-\lambda I_n).
		\end{equation}
	\end{lemma}
	
	\begin{proof}
		Set $Q:=I_n-E$ and $Y:=X\circ E-\lambda I_n$. Since $QYQ=-\lambda Q$ and $\lambda\ne0$, we have $\rank Q=\rank(QYQ)\le\rank Y$. Also,
		\[
		X-\lambda I_n=(I_n+Q)Y(I_n+Q)+Q(X+3\lambda I_n)Q.
		\]
		Consequently,
		\[
		\rank(X-\lambda I_n) \le\rank Y+\rank Q\le 2\rank Y,
		\]
		which proves \eqref{eq:rank-step}.
		
		For \eqref{eq:rank-product}, let $E_1,\ldots,E_m$ be the factors of an idempotent word for $A$. Since $E_1-\lambda I_n$ is invertible,
		successive applications of \eqref{eq:rank-step} give
		\[
		n=\rank(E_1-\lambda I_n)
		\le 2^{m-1}\rank(A-\lambda I_n).
		\]
	\end{proof}
	
	\begin{proof}[Proof of Theorem~\ref{thm:matrices}]
		Let $A\in M_n(\FF)$ be nonscalar. By similarity invariance, we may assume that $A$ is in rational canonical form. Then at least one cyclic summand has order at least two, as otherwise all invariant factors would be linear, and their divisibility would force them to be identical, making $A$ scalar.
		
		Collect the cyclic summands of order at least two into a block $X\in M_r(\FF)$, where $r\ge2$. By Propositions~\ref{prop:exact-two} and~\ref{prop:cyclic}, together with Lemma~\ref{lem:products}, $X$ has a word of length at most $7$. The remaining summands, if any, have order one and hence form a diagonal block.
		
		Proposition~\ref{prop:exact-two}, together with the identity case, gives $\ell(\diag(\delta,1))\le 5$ for every $\delta\in\FF$. Hence
		\[
		\ell(A)\stackrel{\text{Lemma }\ref{lem:padding}}{\le}7+5\ceil{\log_2\frac nr}
		\le 7+5\left(\ceil{\log_2n}-1\right)
		=2+5\ceil{\log_2n}.
		\]
		Together with Lemma~\ref{lem:scalars}, this proves \eqref{eq:matrix-classification} and the upper bound in \eqref{eq:matrix-width}.
		
		If $\FF$ is quadratically closed, Corollary~\ref{cor:two-by-two} gives $w(M_2(\FF)^+)=3$ and permits $c=3$ in Lemma~\ref{lem:padding}. Thus
		\[
		\ell(A)\stackrel{\text{Lemma }\ref{lem:padding}}{\le}
		7+3\ceil{\log_2\frac nr}\le4+3\ceil{\log_2n}.
		\]
		Finally, choose $\lambda\in\FF\setminus\{0,1\}$, which is possible since $\operatorname{char}\FF\ne2$, and put $B:=\lambda I_n+E_{12}$. Since $\rank(B-\lambda I_n)=1$, Lemma~\ref{lem:rank} forces every idempotent word for $B$ to have length at least $1+\ceil{\log_2n}$. This proves the lower bound in \eqref{eq:matrix-width}.
	\end{proof}
	
	\begin{remark}\label{rem:pointwise-lower}
		The proof gives, more generally, for every Jordan subalgebra
		$\mathcal K\subseteq M_n(\FF)^+$ with unit $I_n$,
		\[
		\ell_{\mathcal K}(A)
		\ge1+\ceil{\log_2\frac{n}{\rank(A-\lambda I_n)}},
		\qquad
		A\in\PP(\mathcal K),\quad \lambda\in\FF\setminus\{0,1\}.
		\]
		Here $\rank(A-\lambda I_n)\ne0$ by Lemma~\ref{lem:scalars}.
	\end{remark}
	
	\section{Symmetric matrices}\label{sec:symmetric}
	
	Throughout this section, $\FF$ is quadratically closed of characteristic different from $2$. We work in the self-adjoint operator model $H(V)$ from \eqref{eq:selfadjoint-model}. For $v\in V$, let $v^*:V\to\FF$ be the linear functional defined by $v^*(u):=b(v,u)$. Thus $vv^*\in\End_\FF(V)$ denotes the operator $u\mapsto b(v,u)v$. If $q(v)\ne0$, the operator
	\begin{equation}\label{eq:rank-one-projection}
		P_v:V\to V,
		\qquad
		P_v(x):=\frac{b(v,x)}{q(v)}\,v,
	\end{equation}
	is the self-adjoint idempotent with image $\FF v$ and kernel
	$(\FF v)^\perp$.
	
	\begin{lemma}\label{lem:anisotropic-zero}
		Let $n:=\dim V\ge3$ and let $A\in H(V)$ be nonscalar. Then there exists $m\in V$ with $q(m)\ne0$ and $b(Am,m)=0$.
	\end{lemma}
	
	\begin{proof}
		Define the quadratic form
		\[
		Q(v):=b(Av,v),\qquad v\in V.
		\]
		Since $A$ is self-adjoint, its polar form satisfies $b_Q(u,v)=b(Au,v)$. By the nondegeneracy of $b$, $Q=0$ would imply $A=0$, while $Q=\kappa q$ for some $\kappa\in\FF$ would imply
		$A=\kappa I$. Since $A$ is nonscalar, $Q$ is therefore nonzero and is not a scalar multiple of $q$.
		
		Assume, for contradiction, that $q(v)=0$ whenever $Q(v)=0$.
		
		Suppose first that $\rank Q\le 2$. Since $Q\ne0$, its rank is one or two. As $\FF$ is quadratically closed, after diagonalization and rescaling the basis vectors, $Q$ has the form $\xi_1^2$ or $\xi_1^2+\xi_2^2$, where $\xi_1,\ldots,\xi_n$ are the coordinates in the resulting basis. In the latter case $-1$ is a square in $\FF$, so in either case the zero set of $Q$ contains a hyperplane $U \subseteq V$. By our assumption, $q$ vanishes on $U$, so polarization shows that $b(u,u')=0$ for all $u,u'\in U$, and hence $U\subseteq U^\perp$. This is impossible, since $b$ is nondegenerate, $\dim U=n-1\ge2$, and $\dim U^\perp=1$.
		
		Now let $r:=\rank Q\ge3$. Since $\FF$ is quadratically closed, there is a basis $\{e_1,\ldots,e_n\}$ of $V$ such that $Q\bigl(\sum_{i=1}^n\xi_i e_i\bigr)=\xi_1^2+\cdots+\xi_r^2$ for all $\xi_1,\ldots,\xi_n\in\FF$. Set 
		\[
		U:=\Span\{e_1,\ldots,e_r\}, \qquad R:=\rad Q.
		\]
		Then
		$V=U\oplus R$. Choose $\iota\in\FF$ with $\iota^2=-1$. For distinct $i,j\le r$, the vectors $e_i\pm\iota e_j$ are zeros of $Q$, and hence also of $q$. Therefore
		\[
		0=q(e_i\pm\iota e_j)
		=q(e_i)-q(e_j)\pm2\iota b(e_i,e_j).
		\]
		It follows that $b(e_i,e_j)=0$ and $q(e_i)=q(e_j)$. Thus, with $\kappa:=q(e_1)$, we have $q|_U=\kappa Q|_U$.
		
		For $z\in R$, we have $Q(z)=b_Q(z,z)=0$, so our assumption gives $q(z)=0$. Hence $q$ vanishes on $R$, and polarization shows that $b(z,z')=0$ for all $z,z'\in R$. If $u\in U$ satisfies $Q(u)=0$ and $z\in R$, then $Q(u+z)=0$ because $R=\rad Q$. Hence $q(u)=q(z)=q(u+z)=0$, and therefore $b(u,z)=0$. The vectors $e_i\pm\iota e_j$, with $i\ne j$, span $U$, since for every $i$ we may choose $j\ne i$ and their sum is $2e_i$. Thus every vector in $U$ is orthogonal to every vector in $R$. Since $V=U\oplus R$, every vector in $R$ is orthogonal to $V$. The nondegeneracy of $b$ now gives $R=0$.
		
		Consequently $V=U$ and $q=\kappa Q$. Since $b$ is nondegenerate,
		$q\ne0$, so $\kappa\ne0$. Hence $Q=\kappa^{-1}q$, contradicting the
		choice of $Q$. Therefore there exists $m\in V$ such that $Q(m)=0$ and
		$q(m)\ne0$, that is, $b(Am,m)=0$ and $q(m)\ne0$.
	\end{proof}
	
	\begin{proposition}\label{prop:selfadjoint-reduction}
		Let $A\in H(V)$ be nonscalar and $\dim V\ge3$. There exist
		$B,E\in H(V)$ with $E^2=E$, a scalar $\rho\in\FF\setminus\{0\}$, a
		vector $k\in V$ with $q(k)\ne0$, and $B_0\in H(k^\perp)$ such that
		\[
		A=B\circ E,
		\qquad
		B=B_0\oplus[\rho],
		\]
		where the second identity is taken relative to the decomposition
		$V=k^\perp\oplus_b\FF k$.
	\end{proposition}
	
	\begin{proof}
		Choose $m$ as in Lemma~\ref{lem:anisotropic-zero} and, since $\FF$ is quadratically closed, rescale it so that $q(m)=1$. Put 
		\[
		P:=P_m, \qquad E:=I-P, \qquad  \text{and} \qquad  Z:=A+PA+AP.
		\]
		Then $P,E,Z\in H(V)$ and $E^2=E$. Since $b(Am,m)=0$, we have $PAP=0$; hence $ZE+EZ=2A$ and $PE=EP=0$. Therefore
		\begin{equation}\label{eq:selfadjoint-factor}
			(Z+\tau P)\circ E=A,
			\qquad \tau\in\FF.
		\end{equation}
		It remains to choose $\tau$ so that $Z+\tau P$ has an anisotropic eigenvector with nonzero eigenvalue. For $\sigma\in\FF$ such that $I-\sigma Z$ is invertible, put $h_\sigma:=(I-\sigma Z)^{-1}m$ and set $d(\sigma):=\det(I-\sigma Z)$. Since $d(0)=1$, the polynomial $d$ is nonzero. By the adjugate formula, the vector $d(\sigma)h_\sigma$ has polynomial coordinates in $\sigma$. Consequently, $d(\sigma)b(m,h_\sigma)$ and $d(\sigma)^2q(h_\sigma)$, initially defined when $d(\sigma)\ne0$, extend to polynomials in $\sigma$. Both extensions take the value $1$ at $\sigma=0$, so neither is the zero polynomial. Thus only finitely many values of $\sigma$ are excluded by the conditions that $I-\sigma Z$ be invertible and $b(m,h_\sigma)q(h_\sigma)\ne0$. Since $\FF$ is infinite, we may therefore choose $\sigma\in\FF\setminus\{0\}$ satisfying these conditions. Set
		\[
		\rho:=\sigma^{-1},
		\qquad
		k:=-\sigma h_\sigma.
		\]
		Then $(Z-\rho I)k=m$, while $b(m,k)\ne0$ and $q(k)\ne0$.
		
		Set $B:=Z-P/b(m,k)$. Then $B\in H(V)$. Since $q(m)=1$, we have $Pk=b(m,k)m$, and hence $Bk=Zk-m=\rho k$. Thus $k^\perp$ is $B$-invariant, since $B$ is self-adjoint. Moreover, $q(k)\ne0$ gives $V=k^\perp\oplus_b\FF k$. Therefore, with $B_0:=B|_{k^\perp}\in H(k^\perp)$, $B=B_0\oplus[\rho]$. Finally, taking $\tau=-1/b(m,k)$ in \eqref{eq:selfadjoint-factor} gives $A=B\circ E$.
	\end{proof}
	
	In dimension three, the two-dimensional block $B_0$ may be nonzero with $B_0^2=0$. The following construction handles that case directly.
	
	\begin{lemma}\label{lem:nilpotent-plus-scalar}
		Let $U\subseteq V$ be a two-dimensional subspace, and let $m\in V$ satisfy $q(m)=1$ and $V=U\oplus_b\FF m$. Let $N\in H(U)$ be nonzero with $N^2=0$. Then, for every $\mu\in\FF$,
		\[
		N\oplus[\mu]\in\PP(H(V)),
		\qquad
		\ell_{H(V)}(N\oplus[\mu])\le 3.
		\]
		If $\mu=1$, then $\ell_{H(V)}(N\oplus[1])\le 2$.
	\end{lemma}
	
	\begin{proof}
		Since $N\ne0$, $N^2=0$ and $\dim U=2$, the operator $N$ has rank one. Choose $0\ne u\in\im N$. Then $u=Nx$ for some $x\in U$, so
		\[
		q(u)=b(Nx,Nx)=b(x,N^2x)=0.
		\]
		Thus $u$ is isotropic. Since $\im N=\FF u$, there is a nonzero linear functional $\lambda:U\to\FF$ such that $Nx=\lambda(x)u$ for all $x\in U$. For $x,y\in U$, self-adjointness gives
		\[
		\lambda(x)b(u,y)=b(Nx,y)=b(x,Ny)=\lambda(y)b(u,x).
		\]
		Since $U$ is nondegenerate and $u\ne0$, choose $y\in U$ with $b(u,y)\ne0$ and set
		\[
		\alpha:=\frac{\lambda(y)}{b(u,y)}.
		\]
		The displayed identity gives $\lambda(x)=\alpha b(u,x)$ for all $x\in U$. Since $\lambda\ne0$, we have $\alpha\ne0$, and hence $Nx=\alpha b(u,x)u$ for all $x\in U$. Since $\FF$ is quadratically closed, choose a nonzero $r\in\FF$ with $r^2=\alpha$ and replace $u$ by $ru$; then $Nx=b(u,x)u$ for all $x\in U$. Extend $\{u\}$ to a hyperbolic basis $\{u,v\}$ of $U$.
		
		Choose $\iota\in\FF$ with $\iota^2=-1$. Since
		\[
		q(m+\iota u)=q(m-\iota u)=b(m+\iota u,m-\iota u)=1,
		\]
		we have
		\[
		Y:=P_m+uu^*
		\stackrel{\eqref{eq:rank-one-projection}}{=}
		P_{m+\iota u}\circ P_{m-\iota u}.
		\]
		The operator $uu^*$ restricts to $N$ on $U$ and vanishes on $\FF m$, so $Y=N\oplus[1]$. This proves the case $\mu=1$. Since $I-P_m$ is a self-adjoint idempotent and $Y\circ(I-P_m)=N\oplus[0]$, the case $\mu=0$ follows as well.
		
		Now suppose $\mu\notin\{0,1\}$. We shall choose $c\in V$ with $q(c)=1$ so that $Yc-\mu c$ is a nonzero isotropic vector orthogonal to $c$. Choose $\beta\in\FF\setminus\{0\}$ with $\beta^2=\mu(1-\mu)$, and put
		\[
		c:=\frac{1-\mu^2}{2\beta}u+\beta v+\mu m.
		\]
		Since $b(u,c)=\beta$ and $b(m,c)=\mu$, we have $Yc=\beta u+\mu m$. Using $q(u)=q(v)=0$, $b(u,v)=1$, the orthogonality of $U$ and $\FF m$, $q(m)=1$, and $\beta^2=\mu(1-\mu)$, we obtain
		\[
		q(c)=1,
		\qquad
		b(c,Yc)=\beta^2+\mu^2=\mu,
		\qquad
		q(Yc)=\mu^2.
		\]
		Hence $w:=Yc-\mu c$ satisfies $q(w)=0$ and $b(c,w)=0$. Moreover, $w\ne0$, since its coefficient of $v$ is $-\mu\beta$. Put $k:=c+w/(2\mu)$. Then $q(k)=1$ and $b(k,w)=0$. Since $q(c)=1$, $Y$ is self-adjoint, $c=k-w/(2\mu)$, and $Yc=\mu k+w/2$, \eqref{eq:rank-one-projection} gives
		\[
		Y\circ P_c
		=\frac12\left((Yc)c^*+c(Yc)^*\right)
		=\mu P_k-\frac{1}{4\mu}ww^*.
		\]
		Since $q(k)=1$, we have 
		\[
		V=k^\perp\oplus_b\FF k.
		\]
		The restriction of $b$ to $k^\perp$ is nondegenerate and $\dim k^\perp=2$. The operator $Y\circ P_c$ acts as multiplication by $\mu$ on $\FF k$ and its restriction to $k^\perp$ is the nonzero self-adjoint operator
		\[
		-\frac{1}{4\mu}ww^*,
		\]
		whose square is zero. Choose $\gamma\in\FF$ with $\gamma^2=-1/(4\mu)$ and put
		$a:=\gamma w$. Then $a$ is a nonzero isotropic vector in $k^\perp$ and $aa^*=-ww^*/(4\mu)$. Extend $\{a\}$ to a hyperbolic basis $\{a,d\}$ of $k^\perp$. Since the bases $(a,d,k)$ and $(u,v,m)$ have the same Gram matrix, the linear map $S:V\to V$ determined by
		$Sa=u$, $Sd=v$ and $Sk=m$ is an isometry. Therefore
		\[
		S(Y\circ P_c)S^{-1}=\mu P_m+uu^*=N\oplus[\mu].
		\]
		Conjugation by $S$ is a Jordan automorphism of $H(V)$, so conjugating
		the three idempotent factors $P_{m+\iota u}$, $P_{m-\iota u}$ and $P_c$ by $S$ gives an idempotent word of length three for
		$N\oplus[\mu]$.
	\end{proof}
	
	\begin{theorem}\label{thm:symmetric}
		Let $\FF$ be quadratically closed with $\operatorname{char}\FF\ne2$, and let $n\ge3$. Then
		\[
		\PP(H_n(\FF))=(H_n(\FF)\setminus\FF I_n)\cup\{0,I_n\},
		\]
		and
		\[
		1+\ceil{\log_2n}\le w(H_n(\FF))\le 4n-4.
		\]
		Furthermore, every nonscalar diagonalizable $A\in H_n(\FF)$ satisfies
		\begin{equation}\label{eq:symmetric-diagonalizable}
			\ell_{H_n(\FF)}(A)\le 3\ceil{\log_2n}.
		\end{equation}
	\end{theorem}
	
	\begin{proof}
		We work in $H(V)$, where $\dim V=n$. The scalar obstruction is
		provided by Lemma~\ref{lem:scalars}. To prove
		\eqref{eq:symmetric-diagonalizable}, let $A\in H(V)$ be nonscalar and diagonalizable. Its distinct eigenspaces are mutually orthogonal. Each eigenspace is nondegenerate: a vector in one eigenspace that is orthogonal to it is orthogonal to every eigenspace, and hence to all of $V$. Since $\FF$ is quadratically closed, each eigenspace admits an orthonormal basis. Their union therefore gives an orthogonal diagonalization of $A$. Choose two unequal diagonal entries. By Corollary~\ref{cor:two-by-two}, the corresponding $2\times2$ block has a word of length at most three. Apply Lemma~\ref{lem:padding}, using \eqref{eq:diagonal-two}, to the remaining coordinates. This gives
		\[
		\ell(A)\stackrel{\eqref{eq:padding}}{\le}3+3\ceil{\log_2(n/2)}=3\ceil{\log_2n}.
		\]
		Now let $A\in H(V)$ be nonscalar. We prove the general statement by induction on $n$. Whenever Proposition~\ref{prop:selfadjoint-reduction} is applied, we may, since $\FF$ is quadratically closed, rescale $k$ so that $q(k)=1$; this does not change $k^\perp$ or $\FF k$. Since $q(k)\ne0$, the restriction of $b$ to $k^\perp$ is nondegenerate. As $\FF$ is quadratically closed, $k^\perp$ admits an orthonormal basis, which together with $k$ gives an orthonormal basis of $V$ adapted to $V=k^\perp\oplus_b\FF k$. Thus $H(k^\perp)$ identifies with $H_{n-1}(\FF)$.
		
		For $n=3$, Proposition~\ref{prop:selfadjoint-reduction} gives $A=B\circ E$ with $B=B_0\oplus[\rho]$, $\rho\ne0$. If $B_0$ is nonscalar and $B_0^2\ne0$, Corollary~\ref{cor:two-by-two} gives $\ell(B_0)\le4$. Lemma~\ref{lem:padding}, using \eqref{eq:diagonal-two}, then gives $\ell(B)\le7$, and appending the final factor $E$ gives $\ell(A)\le8$. If $B_0\ne0$ and $B_0^2=0$, Lemma~\ref{lem:nilpotent-plus-scalar}, applied with $U=k^\perp$, $m=k$, $N=B_0$, and $\mu=\rho$, gives at most $3+1=4$ factors. If $B_0$ is scalar and $B$ is nonscalar, then $B$ is diagonalizable, so \eqref{eq:symmetric-diagonalizable} gives $\ell(B)\le 6$ and hence $\ell(A)\le 7$ after appending $E$. If $B$ is scalar, then $B=\rho I$, so $A=\rho E$ is nonscalar and diagonalizable. Hence \eqref{eq:symmetric-diagonalizable} gives $\ell(A)\le 6$. Thus $w(H_3(\FF))\le 8$.
		
		Let $n\ge4$, and apply Proposition~\ref{prop:selfadjoint-reduction} again. If $B_0$ is nonscalar, the induction hypothesis and Lemma~\ref{lem:padding}, using \eqref{eq:diagonal-two}, give
		\[
		\ell(B)\stackrel{\eqref{eq:padding}}{\le} \left(4(n-1)-4\right)+3=4n-5,
		\]
		and hence $\ell(A)\le 4n-4$ after appending $E$. If $B_0$ is scalar and $B$ is nonscalar, then $B$ is diagonalizable, so
		\[
		\ell(A)\le\ell(B)+1\stackrel{\eqref{eq:symmetric-diagonalizable}}{\le}3\ceil{\log_2n}+1.
		\]
		If $B$ is scalar, then $B=\rho I$, so $A=\rho E$ is nonscalar and diagonalizable. Hence
		\[
		\ell(A)\stackrel{\eqref{eq:symmetric-diagonalizable}}{\le}3\ceil{\log_2n}.
		\]
		Thus in either case 
		\[
		\ell(A)\le 3\ceil{\log_2n}+1\le 4n-4.
		\]
		Here $\ceil{\log_2n}\le n-1$ and $3n-2\le 4n-4$. This completes the induction and proves the upper bound. Finally, choose $\lambda\in\FF\setminus\{0,1\}$ and put $C:=\lambda I_n+E_{11}$. The matrix $C$ is nonscalar and diagonalizable, so $C\in\PP(H_n(\FF))$. If $m=\ell_{H_n(\FF)}(C)$, then
		\[
		n\stackrel{\eqref{eq:rank-product}}{\le} 2^{m-1}\rank(C-\lambda I_n) =2^{m-1}.
		\]
		Hence $m\ge1+\ceil{\log_2n}$, proving the lower bound.
	\end{proof}
	
	Therefore, the number of factors is not bounded independently of $n$, although these bounds do not determine the growth order of $w(H_n(\FF))$.
	
	\section{The remaining simple algebras}\label{sec:remaining}
	
	So far we have treated spin factors, full matrix algebras, and symmetric matrices. It remains to consider the symplectic Hermitian and split Albert algebras. We begin with the symplectic Hermitian case from \eqref{eq:symplectic-model}. The argument works over any field of characteristic different from $2$. The following result is due to Waterhouse~\cite[Theorem~3]{Waterhouse2005}. We include a self-contained proof in the present notation.
	
	\begin{lemma}\label{lem:symplectic-normal-form}
		For every $X\in H_{2n}(\FF,\mathrm{sp})$, there exist a symplectic matrix $S\in M_{2n}(\FF)$ and a matrix $A\in M_n(\FF)$ such that
		\begin{equation}\label{eq:symplectic-normal-form}
			S^{-1}XS=\diag(A,A^t).
		\end{equation}
	\end{lemma}
	
	\begin{proof}
		Define the bilinear form
		\[
		\omega:\FF^{2n}\times\FF^{2n}\to\FF, \qquad \omega(u,v):=u^tR_nv.
		\]
		By the definition of $R_n$, the form $\omega$ is alternating and nondegenerate, and \eqref{eq:symplectic-model} says that $X$ is self-adjoint with respect to $\omega$. For every $u\in\FF^{2n}$, the subspace $\Span\{u,Xu,X^2u,\ldots\}$ is totally isotropic. Indeed, for all $i,j\ge0$,
		\[
		\omega(X^iu,X^ju)=\omega(u,X^{i+j}u)=-\omega(X^{i+j}u,u)=-\omega(u,X^{i+j}u),
		\]
		so $\omega(X^iu,X^ju)=0$ since $\operatorname{char}\FF\ne2$.
		
		Let $d$ be the degree of the minimal polynomial of $X$. By rational canonical form, there exists $u\in\FF^{2n}$ such that $u,Xu,\ldots,X^{d-1}u$ are linearly independent. Set
		\[
		U:=\Span\{u,Xu,\ldots,X^{d-1}u\}.
		\]
		Since $\omega$ is nondegenerate, the linear functionals $y\mapsto\omega(X^iu,y)$, $0\le i<d$, are linearly independent. Hence there exists $v\in\FF^{2n}$ such that $\omega(X^iu,v)=\delta_{i,d-1}$ for $0\le i<d$. Set
		\[
		W:=\Span\{v,Xv,\ldots,X^{d-1}v\}.
		\]
		Since the minimal polynomial of $X$ has degree $d$, both $U$ and $W$ are $X$-invariant, and by the first paragraph they are totally isotropic. Moreover, $\omega(X^iu,X^jv)=\omega(X^{i+j}u,v)$. Hence this value is zero when $i+j<d-1$ and is equal to $1$ when $i+j=d-1$. Thus, after reversing the columns, the matrix
		$(\omega(X^iu,X^jv))_{0\le i,j<d}$ is triangular with diagonal entries equal to $1$. It follows that $v,Xv,\ldots,X^{d-1}v$ are linearly independent and that the pairing between $U$ and $W$ is nondegenerate. Consequently, $U\cap W=0$ and the restriction of $\omega$ to $U\oplus W$ is nondegenerate.
		
		The subspace of vectors $\omega$-orthogonal to $U\oplus W$ is $X$-invariant, since $X$ is self-adjoint and $U\oplus W$ is $X$-invariant. Since the restriction of $\omega$ to $U\oplus W$ is nondegenerate, its orthogonal complement is also nondegenerate. We may therefore repeat the construction inside that complement. At each stage, the new pair of totally isotropic subspaces is orthogonal to all previously constructed pairs. Hence the direct sums of the first components and of the second components remain totally isotropic and $X$-invariant. Continuing until the orthogonal complement is zero yields $X$-invariant totally isotropic subspaces $L$ and $M$ such that $\FF^{2n}=L\oplus M$. At each stage the two added summands have the same dimension, so $\dim L=\dim M=n$.
		
		Choose a basis $\{e_1,\ldots,e_n\}$ of $L$. Since the pairing between $L$ and $M$ induced by $\omega$ is nondegenerate, there is a basis $\{f_1,\ldots,f_n\}$ of $M$ such that $\omega(e_i,f_j)=\delta_{ij}$. Since $\FF^{2n}=L\oplus M$, the vectors $e_1,\ldots,e_n,f_1,\ldots,f_n$ form a basis of $\FF^{2n}$. As $L$ and $M$ are totally isotropic, the matrix of $\omega$ with respect to the ordered basis $(e_1,\ldots,e_n,f_1,\ldots,f_n)$ is $R_n$.
		
		Let $S\in M_{2n}(\FF)$ be the matrix whose columns, with respect to the standard basis of $\FF^{2n}$, are $e_1,\ldots,e_n,f_1,\ldots,f_n$. Then $S^tR_nS=R_n$, so $S$ is symplectic. Since $L$ and $M$ are $X$-invariant, the matrix of $X$ in this basis has the form $S^{-1}XS=\diag(A,B)$ for some $A,B\in M_n(\FF)$. The self-adjointness of $X$ with respect to $\omega$ gives $\diag(A,B)^tR_n=R_n\diag(A,B)$, and hence $B=A^t$. This proves
		\eqref{eq:symplectic-normal-form}.
	\end{proof}
	
	\begin{theorem}\label{thm:symplectic}
		Let $\FF$ be a field with $\operatorname{char}\FF\ne2$, and let $n\ge2$. Then
		\[
		\PP(H_{2n}(\FF,\mathrm{sp}))=\left(H_{2n}(\FF,\mathrm{sp})\setminus\FF I_{2n}\right)\cup\{0,I_{2n}\},
		\]
		and
		\begin{equation}\label{eq:symplectic-width}
			1+\ceil{\log_2n} \le w(H_{2n}(\FF,\mathrm{sp}))  \le w(M_n(\FF)^+).
		\end{equation}
		In particular, the maximum length grows logarithmically with $n$.
	\end{theorem}
	\begin{proof}
		The map
		\[
		\Phi:M_n(\FF)^+\to H_{2n}(\FF,\mathrm{sp}), \qquad \Phi(A):=\diag(A,A^t),
		\]
		is an injective unital Jordan homomorphism. Let
		$X\in H_{2n}(\FF,\mathrm{sp})$ be nonscalar. By
		Lemma~\ref{lem:symplectic-normal-form}, there exist a symplectic matrix $S\in M_{2n}(\FF)$ and a nonscalar matrix $A\in M_n(\FF)$ satisfying \eqref{eq:symplectic-normal-form}. By Theorem~\ref{thm:matrices}, $A$ has an idempotent word of length at most $w(M_n(\FF)^+)$. Applying $\Phi$ to its factors and conjugating them by $S$ gives a word for $X$ of the same length. Together with Lemma~\ref{lem:scalars}, this proves the asserted classification and the upper bound in \eqref{eq:symplectic-width}.
		
		For the lower bound, choose $\lambda\in\FF\setminus\{0,1\}$ and set
		$X:=\Phi(\lambda I_n+E_{12})$. Then
		$\rank(X-\lambda I_{2n})=2$. If $X$ has an idempotent word of length
		$m$ in $H_{2n}(\FF,\mathrm{sp})$, then, viewing the same word in
		$M_{2n}(\FF)^+$,
		\[
		2n\stackrel{\eqref{eq:rank-product}}{\le} 2^{m-1}\rank(X-\lambda I_{2n})=2^{m-1}\cdot2.
		\]
		Hence $m\ge1+\ceil{\log_2n}$.
	\end{proof}
	
	Over quadratically closed fields, Theorem~\ref{thm:spin}\textup{(a)}, applied to the spin-factor identification in Section~\ref{subsec:models}, gives $w(H_4(\FF,\mathrm{sp}))=3$.
	
	We now determine the idempotent words in the split Albert algebra $\mathcal A$ from \eqref{eq:albert-model}. We use the following embedding theorem, due to Jacobson in characteristic different from $2$ \cite[Theorem~IX.11]{Jacobson1968}. Petersson later extended it to arbitrary characteristic \cite[Theorem~6.3]{Petersson2015}.
	
	\begin{theorem}[Embedding theorem]\label{thm:albert-embedding}
		Let $\mathcal A$ be a split Albert algebra over $\FF$. Every $x\in\mathcal A$ is contained in a Jordan subalgebra $\mathcal B\subseteq\mathcal A$ with the same unit as $\mathcal A$ and isomorphic to $M_3(\FF)^+$.
	\end{theorem}
	
	\begin{proposition}\label{prop:albert}
		Let $\FF$ be a field with $\operatorname{char}\FF\ne2$. For the split Albert algebra $\mathcal A$, we have
		\[
		\PP(\mathcal A)=(\mathcal A\setminus\FF1)\cup\{0,1\},
		\qquad w(\mathcal A)\le 10.
		\]
		If $\FF$ is quadratically closed, then $w(\mathcal A)\le 7$.
	\end{proposition}
	
	\begin{proof}
		Let $x\in\mathcal A$ be nonscalar. By Theorem~\ref{thm:albert-embedding}, $x$ belongs to a Jordan subalgebra $\mathcal B\subseteq\mathcal A$ with the same unit as $\mathcal A$ and isomorphic to $M_3(\FF)^+$. Hence $x$ is nonscalar in $\mathcal B$. Let $X\in M_3(\FF)$ be its image under such an isomorphism.
		
		If $X$ is cyclic, Proposition~\ref{prop:cyclic} gives an idempotent word of length at most seven. Otherwise, $X$ has at least two invariant factors. Since their positive degrees sum to $3$, the only possibilities are $1+2$ and $1+1+1$. The latter cannot occur, since divisibility would then force the three monic linear invariant factors to coincide and make $X$ scalar. Hence the degrees are $1+2$, and rational canonical form shows that $X$ is similar to a block diagonal matrix $\diag(C,[\lambda])$, where $C\in M_2(\FF)$ is cyclic and $\lambda\in\FF$. By similarity invariance, Proposition~\ref{prop:exact-two} and Lemma~\ref{lem:padding} give an idempotent word of length at most ten, or at most six if $\FF$ is quadratically closed. Thus every nonscalar element of $\mathcal A$ has an idempotent word with the stated bound. Finally, $0$ and $1$ are idempotents, while Lemma~\ref{lem:scalars} excludes every other scalar element.
	\end{proof}
	
	\begin{proof}[Proof of Theorem~\ref{thm:classification}]
		By Theorem~\ref{thm:structure}, it suffices to consider the six families listed there. Lemma~\ref{lem:scalars} handles the scalar algebra. Theorem~\ref{thm:spin} treats all spin factors. Since $\FF$ is algebraically closed, every nondegenerate symmetric bilinear form on a
		two-dimensional $\FF$-vector space admits an orthonormal basis. Hence every spin factor with $\dim V=2$ is isomorphic to $H_2(\FF)$ via the spin-factor identification in the proof of
		Corollary~\ref{cor:two-by-two}; for these, that corollary gives the description and the exact four-factor statement. The full matrix, symmetric matrix and symplectic Hermitian cases follow from Theorems~\ref{thm:matrices}, \ref{thm:symmetric}, and \ref{thm:symplectic}, respectively. Proposition~\ref{prop:albert} completes the split Albert case.
	\end{proof}
	
	\begin{corollary}\label{cor:arbitrary-bracketing}
		Let $\FF$ be algebraically closed with $\operatorname{char}\FF\ne2$, and let
		$\J$ be a finite-dimensional simple Jordan algebra over $\FF$. If
		$\dim\J>1$, then every element of $\J$ is the value of a product of
		idempotents with some bracketing. If $\dim\J=1$, the only such values are
		$0$ and $1$.
	\end{corollary}
	
	\begin{proof}
		The one-dimensional case is immediate, so assume that $\dim\J>1$. By Theorem~\ref{thm:classification}, the only elements not already represented by left-nested idempotent words are the nontrivial scalar elements and, when $\J\cong H_2(\FF)$, the nonzero square-zero elements.
		
		By Theorem~\ref{thm:structure}, $\J$ contains a nontrivial idempotent $e$. Let $\lambda\in\FF\setminus\{0,1\}$. The elements
		\[
		1+e
		\qquad\text{and}\qquad
		\lambda\left(1-\frac12e\right)
		\]
		are nonscalar. If $\J\cong H_2(\FF)$, neither is square-zero, since
		\[
		(1+e)^2=1+3e, \qquad
		\left(\lambda\left(1-\frac12e\right)\right)^2  =\lambda^2\left(1-\frac34e\right),
		\]
		and both squares are nonzero because $e$ is nontrivial. Hence Theorem~\ref{thm:classification} gives left-nested idempotent words with these two values. Bracketing the two words together yields
		\[
		(1+e)\circ\lambda\left(1-\frac12e\right)=\lambda1.
		\]
		Since $0$ and $1$ are themselves idempotents, this represents every scalar element of $\J$.
		
		It remains to consider a nonzero $N\in H_2(\FF)$ with $N^2=0$. Under the spin-factor identification from the proof of Corollary~\ref{cor:two-by-two}, $H_2(\FF)$ identifies with $\mathcal S(V,b)$, where $V$ is the two-dimensional space of traceless symmetric matrices, and $N$ corresponds to a nonzero isotropic vector $u\in V$. Extend $\{u\}$ to a hyperbolic basis $\{u,z\}$ of $V$. Since $z$ and $u+z$ are nonzero, Theorem~\ref{thm:spin} \textup{(b)} shows that both $1+z$ and $-1+u+z$ are values of left-nested idempotent words. By \eqref{eq:spin-product},
		\[
		(1+z)\circ(-1+u+z)=u.
		\]
		Bracketing these two words together therefore gives a product of idempotents with value $N$.
	\end{proof}
	
	Recall that an ideal $\mathcal I$ of a finite-dimensional Jordan algebra $\J$ over $\FF$ is \emph{nilpotent} if, for some integer $s\ge1$, every product of at least $s$ elements of $\mathcal I$ vanishes, however bracketed. The largest such ideal is the \emph{radical} of $\J$, and $\J$ is called \emph{semisimple} if its radical is zero. Over an algebraically closed field of characteristic different from $2$, $\J$ is semisimple if and only if it is a direct sum of simple ideals; see~\cite[Part~I, Section~2.13, Renaissance Structure Theorem]{McCrimmon2004}. Together with Lemma~\ref{lem:products} and Theorem~\ref{thm:classification}, this decomposition gives the following corollary.
	
	\begin{corollary}\label{cor:semisimple}
		Let $\FF$ be algebraically closed with $\operatorname{char}\FF\ne2$, and let
		$\J$ be a finite-dimensional semisimple Jordan algebra. Write $\J=\J_1\oplus\cdots\oplus\J_s$ as a direct sum of simple ideals. Then
		$x=x_1+\cdots+x_s$, with $x_i\in\J_i$, belongs to $\PP(\J)$ if and only if
		each $x_i$ satisfies the condition in Theorem~\ref{thm:classification}
		for $\J_i$. In that case,
		\[
		\ell_\J(x)=\max_{1\le i \le s}\ell_{\J_i}(x_i),  \qquad   w(\J)=\max_{1 \le i \le s}w(\J_i).
		\]
	\end{corollary}
	
	\begin{remark}\label{rem:quadratic-scope}
		The conclusion of Theorem~\ref{thm:classification} need not hold over a
		quadratically closed field. Indeed, inside an algebraic closure of
		$\mathbb F_5$, set
		\[
		\mathbb E:=\bigcup_{r\ge0}\mathbb F_{5^{2^r}}.
		\]
		Every quadratic polynomial over $\mathbb E$ has its coefficients in some $\mathbb F_{5^{2^r}}$ and splits over $\mathbb F_{5^{2^{r+1}}}$, so $\mathbb E$ is quadratically closed. On the other hand, the cubic polynomial
		$f(t):=t^3+t+1\in\mathbb F_5[t]\subseteq\mathbb E[t]$ has no root in $\mathbb F_5$ and is therefore irreducible over $\mathbb F_5$. If $f$ had a root $\alpha\in\mathbb E$, then
		$\alpha\in\mathbb F_{5^{2^r}}$ for some $r$. Since $f$ is irreducible over $\mathbb F_5$, we would have $[\mathbb F_5(\alpha):\mathbb F_5]=3$, which would have to divide $2^r$, a contradiction. Thus $f$ has no root in $\mathbb E$ and, being cubic, is irreducible over $\mathbb E$.
		
		Let $(f)$ denote the principal ideal generated by $f$ in $\mathbb E[t]$. Since $f$ is irreducible over $\mathbb E$, the quotient $\mathcal K:=\mathbb E[t]/(f)$ is a field extension of $\mathbb E$ of degree $3$. Viewed as a Jordan algebra over $\mathbb E$ with its usual multiplication, $\mathcal K$ is three-dimensional and simple, and its only idempotents are $0$ and $1$. Consequently, $\PP(\mathcal K)=\{0,1\}$, although $\mathcal K\setminus\mathbb E1\ne\emptyset$.
	\end{remark}
	
	\begin{remark}\label{rem:real-closed}
		Let $\FF$ be real closed, endowed with its unique ordering, and let $n\ge2$. Recall that $A\in H_n(\FF)$ is \emph{positive definite} if $v^tAv>0$ for every nonzero $v\in\FF^n$. Let
		$A\in\PP(H_n(\FF))$ with $A\ne I_n$. After deleting trailing unit factors, write $A=X\circ E$ with $E\ne I_n$ an idempotent, taking $X=E=A$ when the representing word has length one. For every $v\in\ker E$,
		\[
		v^tAv=\frac12v^t(XE+EX)v=0.
		\]
		Since $\ker E\ne0$, no nonscalar positive definite matrix belongs to $\PP(H_n(\FF))$. In particular, $\diag(1,2,\ldots,2)\notin\PP(H_n(\FF))$.
	\end{remark}
	
	\subsection*{Funding}
	This research was supported by the European Union -- NextGenerationEU through the National Recovery and Resilience Plan 2021--2026, Institutional Grants of the University of Zagreb Faculty of Science (IK IA 1.1.3. Impact4Math, PMF-CROFUND).
	
	\subsection*{AI use disclosure}
	During the preparation of this manuscript, the authors used OpenAI's GPT-5.6 Sol for preliminary exploration of proof strategies and literature-search assistance, and OpenAI's GPT-6 Astra for suggestions on proof simplification and exposition, as well as for copyediting and proofreading. The authors reviewed and verified all AI-assisted material and take full responsibility for the final content.

\end{document}